\pdfoutput=1
\RequirePackage{ifpdf}
\ifpdf 
\documentclass[pdftex]{sigma}
\else
\documentclass{sigma}
\fi

\usepackage{enumerate}
\usepackage[all]{xy}

\numberwithin{equation}{section}

\newtheorem{THM}{Theorem}[section]

\newtheorem{LEM}[THM]{Lemma}
\newtheorem{PROP}[THM]{Proposition}
\newtheorem{QUE}[THM]{Question}
 { \theoremstyle{definition}
\newtheorem{DEF}[THM]{Definition}

\newtheorem{REM}[THM]{Remark}}

\newcommand{\ra}{\rightarrow}
\newcommand{\LRa}{\Longrightarrow}
\newcommand{\lra}{\longrightarrow}
\newcommand{\p}{\prime}
\newcommand{\Zbb}{\mathbb{Z}}
\newcommand{\Pbb}{\mathbb{P}}
\newcommand{\Cbb}{\mathbb{C}}
\newcommand{\al}{\alpha}
\newcommand{\om}{\omega}
\newcommand{\lam}{\lambda}
\newcommand{\be}{\beta}
\newcommand{\dt}{\delta}
\newcommand{\bt}{\bullet}
\newcommand{\Aut}{\operatorname{Aut}}
\newcommand{\Ac}{\mathcal{A}}
\newcommand{\Oc}{\mathcal{O}}
\newcommand{\Jc}{\mathcal{J}}
\newcommand{\Gc}{\mathcal{G}}
\newcommand{\Xfr}{\mathfrak{X}}
\newcommand{\Ec}{\mathcal{E}}
\newcommand{\Qcl}{\mathcal{Q}}

\begin{document}

\newcommand{\arXivNumber}{1804.06366}

\renewcommand{\PaperNumber}{094}

\FirstPageHeading

\ShortArticleName{Higher Obstructions for Complex Supermanifolds}

\ArticleName{Higher Obstructions of Complex Supermanifolds}

\Author{Kowshik BETTADAPURA}

\AuthorNameForHeading{K.~Bettadapura}

\Address{Yau Mathematical Sciences Center, Tsinghua University, Haidian, Beijing, 100084, China}
\Email{\href{mailto:kowshik@mail.tsinghua.edu.cn}{kowshik@mail.tsinghua.edu.cn}}

\ArticleDates{Received April 29, 2018, in final form August 30, 2018; Published online September 07, 2018}

\Abstract{In this article we introduce the notion of a `good model' in order to study the higher obstructions of complex supermanifolds. We identify necessary and sufficient conditions for such models to exist. Illustrations over Riemann surfaces are provided.}

\Classification{32C11; 58A50}

\section{Introduction}

Complex supermanifolds are spaces modelled on the data of a complex manifold $X$ and holomorphic vector bundle $E\rightarrow X$. Accordingly, we refer to the pair $(X, E)$ as a `model'. Supermanifolds can be either split or non-split. A splitting is a global isomorphism to some fixed model space. Berezin in \cite[pp.~163--169]{BER} observed that to any complex supermanifold, there will be associated a hierarchy of inductively defined cohomology classes representing obstructions to the existence of a splitting.\footnote{In the terminology of Berezin, a splitting is referred to as a `retraction' and a split supermanifold is referred to as `simple'.} These classes are called \emph{obstruction classes to splitting}, or simply `obstructions'. To any supermanifold one can always define its \emph{primary} obstruction. If this vanishes, another will appear in its place, sitting at a higher degree (in a suitable sense). Accordingly, these classes are referred to as `higher obstructions'.

Donagi and Witten in \cite{DW1} observed that the higher obstructions to splitting a supermani\-fold~$\mathfrak X$ might fail to represent a genuine obstruction to splitting~$\mathfrak X$. That is, it might well be that $\mathfrak X$ is split even if it supports an atlas in which a higher obstruction to the existence of a splitting does not vanish. Such a phenomenon seems to be difficult to illustrate in practice. For instance, it was shown by the author in~\cite{BETTSRSDEF} that certain deformations of super Riemann surfaces will \emph{not} support such atlases.

In this article we are motivated by the following question:
\begin{gather}
	 \text{\emph{When does a higher obstruction to splitting $\mathfrak X$ represent a genuine obstruction?}}\label{star}
\end{gather}
Following the terminology of Donagi and Witten, if a higher obstruction is not genuine it is referred to as \emph{exotic}.

We summarise the main ideas and results in this article below.

\subsection{Article summary}
Obstruction classes are associated to supermanifolds, while obstruction spaces are associated to models. In our attempt to address the question in~\eqref{star}, we will look to classify models $(X, E)$. In Definition \ref{rh89hf89jf093j9f0fj03} we introduce the notion of a `good model' $(X, E)$. This is a model on which every higher obstruction for any supermanifold modelled on $(X, E)$ will be genuine. Lemma~\ref{futhirhhiofoirjfoi} clarifies the relation to the splitting problem. Our main result is Theorem~\ref{ruihf8h98f983jfj30f} where we obtain necessary and sufficient conditions for $(X, E)$ to be a good model.

Section \ref{rfh94hf98hf984jf9j90k3} is devoted to illustrations of models which are good and otherwise. We restrict our attention to holomorphic vector bundles on Riemann surfaces. Over the projective line $\mathbb P^1_{\mathbb C}$, when the bundle has rank $3$, a general characterisation of good models is derived in Theorem~\ref{rfh794hf984f98j}. It is based essentially on an example by Donagi and Witten in~\cite{DW1}. This is then extended to higher rank for a particular class of bundles on $\mathbb P^1_{\mathbb C}$ in Theorem~\ref{hf78hf984f8j40f094}.

The relation of higher obstructions to the splitting problem is subtle. In Appendix~\ref{rgf784g87hf98j3fj0j9f33} we submit a proof of the assertion in Theorem~\ref{4fh984fj8jf093j303kf}, being: if the primary obstruction of a~supermanifold $\mathfrak X$ does not vanish, then $\mathfrak X$ is non-split. There is no reason for the analogous statement for higher obstructions to hold. It is for this reason that we introduce variants of splitting such as `strong splitting' and `weak non-splitting' in Definition~\ref{rfh974hf9h389fj3}.

\section{Complex supermanifolds}\label{rfh978f9h3f98hf983}

There are a number of different categories of supermanifolds and each comes with its own level of subtlety as detailed in \cite{BRUZZO}. Those considered in this article fall under the umbrella of the `algebro-geometric' supermanifolds. These are defined as locally ringed spaces which are modelled on a~given space. Standard reference include \cite{QFAS, YMAN}. One can however bypass that language and give an equivalent definition in terms of \v{C}ech cohomology sets. This is the approach we consider here since it is more convenient for the purposes of this article.

\subsection{Green's automorphism groups}
Let $X$ be a complex manifold and $E\rightarrow X$ a holomorphic vector bundle. Throughout this article we will refer to the pair $(X, E)$ as a~\emph{model}. Denote by $\Ec$ the sheaf of local sections of~$E$. We can then form the sheaf of exterior algebras~$\wedge^\bt \Ec$ on~$X$.

\begin{REM}\label{rhf97hf93h8f9j3fj0}
Note that $\wedge^\bt \Ec$ is a sheaf of $\Zbb_2$-graded, supercommutative algebras. It also admits a compatible $\Zbb$-grading. We will however only consider automorphisms of $\wedge^\bt \Ec$ which preserve the $\Zbb_2$-grading. The sheaf of such automorphisms will be denoted $\operatorname{{\mathcal A}ut}_{\Zbb_2}(\wedge^\bt \Ec)$. However we will suppress the subscript `$\Zbb_2$' so as to avoid cumbersome notation and simply write $\operatorname{{\mathcal A}ut}\wedge^\bullet \Ec$.
\end{REM}

Let $\Jc_{(X, E)}\subset \wedge^\bt \Ec$ be the ideal generated by local sections of $\Ec$. As an $\Oc_X$-module, $\Jc_{(X, E)} = \oplus_{j>0}\wedge^j\Ec$. The following sheaf of non-abelian groups were defined by Green in \cite{GREEN}:
\begin{gather}
\Gc_{(X, E)}^{(k)}:=\big\{\al\in \operatorname{{\mathcal A}ut}\wedge^\bt\Ec \,|\, \al(u) - u\in \Jc_{(X, E)}^k \big\}.\label{rhf4hf98498jf8039f33f}
\end{gather}
The \v{C}ech cohomology of $X$ valued in $\Gc_{(X, E)}^{(k)}$ is of particular importance. We present now our working definition of a supermanifold.

\begin{DEF}\label{rfh89hf98hf8030}A \emph{supermanifold modelled on $(X, E)$} is an element of the \v{C}ech cohomology set $\check{H}^1\big(X, \Gc_{(X, E)}^{(2)}\big)$.
\end{DEF}

Generally $\check{H}^1(X, G)$, for $G$ a sheaf of non-abelian groups on $X$, will be a pointed set. The basepoint corresponds to the identity section of $G$. That is, under the inclusion $\{e\}\ra G$, for $\{e\}$ the trivial group, one obtains a map $\check{H}^1(X,\{e\}) \ra \check{H}^1(X,G)$ of pointed sets. Now $\check{H}^1(X,\{e\})$ is a one-point set. The basepoint in $\check{H}^1(X,G)$ is the image of this one-point set.

\begin{DEF}\label{rhrhf8938fj30fj03}
The basepoint in $\check{H}^1\big(X, \Gc_{(X, E)}^{(2)}\big)$ is referred to as \emph{the split supermanifold} model\-led on~$(X, E)$, or simply the \emph{split model}.
\end{DEF}

The 1-cohomology set $\check{H}^1\big(X, \Gc_{(X, E)}^{(2)}\big)$ will be referred to as the \emph{set of supermanifolds modelled on~$(X, E)$}.

\subsection{Isomorphisms of supermanifolds}
We will make use of an early observation of Grothendieck in \cite{GROTHNONAB}, discussed also in \cite[p.~160]{BRY}, regarding sheaves of non-abelian groups.

\begin{THM}\label{groeiveytvdbno}
Let $X$ be a $($paracompact$)$ topological space and suppose the following is a short exact sequence of sheaves of $($not necessarily abelian$)$ groups on~$X$
\begin{gather*}
e\lra A\lra B\lra C\lra e.
\end{gather*}
Then there exists a long exact sequence
\begin{align}
e & \lra H^0(X, A) \lra H^0(X, B) \lra H^0(X, C)\label{rhf894h98ffj093j03}\\
&\lra \check{H}^1(X, A) \lra\check{H}^1(X, B)\lra\check{H}^1(X, C),\label{kfjvbtbrnnvoir}
\end{align}
where the sequence in \eqref{rhf894h98ffj093j03} is as groups whereas that in \eqref{kfjvbtbrnnvoir} is as pointed sets. Furthermore, there exists an action of $H^0(X, C)$ on the set $\check{H}^1(X, A)$ such that the following diagram commutes
\begin{gather*}
\xymatrix{\ar[dr]\check{H}^1(X, A) \ar[rr] & & \check{H}^1(X, B)\\
& \check{H}^1(X, A)/H^0(X, C).\ar@{-->}[ur]&}
\end{gather*}
\end{THM}

We have so far defined supermanifolds as elements of a certain \v{C}ech cohomology set. As for when two supermanifolds are isomorphic, this is defined as follows. Firstly observe for $k = 2$ the following short exact sequence of sheaves of groups
\begin{gather}
\{e\} \lra \Gc^{(2)}_{(X, E)} \lra \operatorname{{\mathcal A}ut}\wedge^\bt\Ec \lra \operatorname{{\mathcal A}ut} \Ec\lra\{e\}.\label{hf984jf8j093f0393}
\end{gather}
Applying Theorem \ref{groeiveytvdbno} gives a long exact sequence on cohomology. Consider the following piece
\begin{align}
\cdots\lra H^0\big(X, \operatorname{{\mathcal A}ut} \Ec\big) \lra \check{H}^1\big(X, \Gc^{(2)}_{(X, E)} \big) \lra \check{H}^1\big(X, \operatorname{{\mathcal A}ut}\wedge^\bt\Ec\big) \lra\cdots.
\label{rjf984jf98jf09jf90f930}
\end{align}
Isomorphisms are defined as follows.

\begin{DEF}\label{rfhhf983fj390f3}Two supermanifolds modelled on $(X, E)$ are \emph{isomorphic} if and only if their image in $\check{H}^1\big(X, \operatorname{{\mathcal A}ut}\wedge^\bt\Ec\big)$ coincide.
\end{DEF}

Note, $H^0(X, \operatorname{{\mathcal A}ut} \Ec)$ coincides with the group $\Aut(E)$ of global automorphisms. From the latter part of Theorem~\ref{groeiveytvdbno} we have an action of the group~$\Aut(E)$ on $\check{H}^1\big(X, \Gc^{(2)}_{(X, E)}\big)$ and a~well-defined map from the orbits to $\check{H}^1\big(X, \operatorname{{\mathcal A}ut}\wedge^\bt\Ec\big)$. This leads to the following result, first described by Green in~\cite{GREEN}, on the general classification of supermanifolds. As per the definitions we have made it follows from Theorem~\ref{groeiveytvdbno}.

\begin{THM}\label{rh9r8h984hf89jf0fj3}There exists a bijective correspondence
\begin{gather*}
\frac{\left\{\text{\rm supermanifolds modelled on $(X, E)$}\right\}}{\text{\rm up to isomorphism}} \cong \frac{\check{H}^1\big(X, \Gc^{(2)}_{(X, E)}\big)}{\Aut(E)}.
\end{gather*}
\end{THM}

\subsection{Obstruction theory}
Recall that in Definition \ref{rhrhf8938fj30fj03} we termed the basepoint in $\check{H}^1\big(X, \Gc^{(2)}_{(X, E)}\big)$ the \emph{split model}. With the notion of isomorphism in Definition~\ref{rfhhf983fj390f3} we have:

\begin{DEF}A supermanifold if said to be \emph{split} if and only if it is isomorphic to the split model.
\end{DEF}

Obstruction theory for supermanifolds is concerned with the following question which is referred to as \emph{the splitting problem}:

\begin{QUE}\label{rfiufh498fj894fj0409}Given a supermanifold, how can one tell if it is split?
\end{QUE}

The splitting problem has a long history dating back to its original formulation by Berezin in \cite{BER}, framed using the term `retraction'. Examples of non-split supermanifolds were provided by Berezin and contemporaries including Palamodov in \cite{PALAM}, Green in \cite{GREEN} and Manin in \cite{YMAN}. Outside of mathematical curiosity however, the relevance of the splitting problem to theoretical physics is not so clear. Recent advances by Donagi and Witten in \cite{DW2, DW1} serve to show that non-splitting is a phenomenon that, at least in superstring theory, cannot be ignored. As our starting point in addressing the splitting problem, we have the following result by Green in \cite{GREEN}.

\begin{PROP}\label{rjf984jf894jf90f903k0}
For each $k\geq 2$ there exists a short exact sequence of sheaves of groups
\begin{gather*}
\{e\} \lra \Gc^{(k+1)}_{(X, E)} \lra \Gc^{(k)}_{(X, E)} \lra \Qcl^{(k)}_{(X, E)} \lra \{e\},
\end{gather*}
where the inclusion $\Gc^{(k+1)}_{(X, E)} \ra \Gc^{(k)}_{(X, E)}$ is normal and the factor $\Qcl^{(k)}_{(X, E)}$ is abelian.
\end{PROP}

An important utility of Proposition~\ref{rjf984jf894jf90f903k0} is in relating non-abelian cohomology to the cohomology of abelian sheaves. By Theorem~\ref{groeiveytvdbno} the short exact sequence in Proposition~\ref{rjf984jf894jf90f903k0} will induce, for each~$k$, a long exact sequence on cohomology. The piece of most relevance for our present purposes is
\begin{gather}
\cdots\lra \check{H}^1\big(X, \Gc^{(k+1)}_{(X, E)}\big) \lra \check{H}^1\big(X, \Gc^{(k)}_{(X, E)}\big) \stackrel{\om_*}{\lra} H^1\big(X, \Qcl^{(k)}_{(X, E)}\big),\label{rfj894jf90jf903k903k0}
\end{gather}
where in the latter we have used the identification of \v{C}ech cohomology and sheaf cohomology for abelian sheaves.

\begin{DEF}\label{rhf94h98h489fj90fk903k}To any model $(X, E)$ we term the following important constructs for each $k\geq 2$:
\begin{itemize}\itemsep=0pt
	\item the sheaf $\Qcl^{(k)}_{(X, E)}$ in Proposition~\ref{rjf984jf894jf90f903k0} is referred to as the $k$-th \emph{obstruction sheaf};
	\item the cohomology group $H^1\big(X, \Qcl_{(X, E)}^{(k)}\big)$ is referred to as the \emph{$k$-th obstruction space}.
\end{itemize}
\end{DEF}

Inspecting \eqref{rfj894jf90jf903k903k0} shows: there will exist a class in the second obstruction space associated to \emph{any} supermanifold $\Xfr$ modelled on $(X, E)$. We denote this class by $\om_*(\Xfr)$.

\begin{DEF}\label{rjfi4fj04kf0kf03kf3}
Let $\Xfr$ be a supermanifold modelled on $(X, E)$. The class $\om_*(\Xfr)$ in the second obstruction space will be referred to as the \emph{primary obstruction} to splitting $\Xfr$.
\end{DEF}

The terminology in the above definition is justified in the following classical result.

\begin{THM}\label{4fh984fj8jf093j303kf}Let $\Xfr$ be a supermanifold. If its primary obstruction is non-vanishing, then $\Xfr$ is non-split.
\end{THM}

\begin{proof}A proof of this theorem, stated using the notion of `retractibility,' can be found in \cite[Theorem 4.6.2, p.~158]{BER}. For completeness, we present a proof based on our formulation of isomorphisms of supermanifolds in Definition \ref{rfhhf983fj390f3}. As the proof is a little involved, we defer it to Appendix~\ref{rgf784g87hf98j3fj0j9f33}.
\end{proof}

Note that Theorem \ref{4fh984fj8jf093j303kf} above cleanly addresses the splitting problem, as posed in Ques\-tion~\ref{rfiufh498fj894fj0409}. If the primary obstruction to splitting vanishes however, it far more subtle to adequately address this problem.

\section{Higher obstructions}

\subsection{Higher atlases}
We are interested in the splitting problem in instances where the primary obstruction vanishes. To that extent we begin with the following string of definitions.

\begin{DEF}An element in $\check{H}^1\big(X, \Gc^{(k)}_{(X, E)}\big)$ is referred to as a \emph{$(k-1)$-split atlas} for a supermanifold modelled on $(X, E)$.
\end{DEF}

\begin{DEF}\label{hf894f894jfj09f30}If a supermanifold $\Xfr$ lies in the image of a $(k-1)$-split atlas $\Xfr^{(k)}$, then $\Xfr$ is said to be \emph{$(k-1)$-split}. Conversely, the image of $\Xfr^{(k)}$ in $\check{H}^1\big(X, \Gc^{(2)}_{(X, E)}\big)$ will be referred to as the \emph{supermanifold associated to $\Xfr^{(k)}$}.
\end{DEF}

Associated to any $(k-1)$-split atlas will be a class in the $k$-th obstruction space, which is the cohomology group $H^1\big(X, \mathcal Q^{(k)}_{(X, E)}\big)$. Following Definition~\ref{rjfi4fj04kf0kf03kf3} we have:

\begin{DEF}\label{rfh479f74hf893jf038jf93} The class in the $k$-th obstruction space associated to any $(k-1)$-split atlas will be referred to as the \emph{primary obstruction of the $(k-1)$-split atlas}.
\end{DEF}

\begin{DEF}If the primary obstruction of a $(k-1)$-split atlas $\Xfr^{(k)}$ is non-vanishing, then~$\Xfr^{(k)}$ is said to be \emph{obstructed}.
\end{DEF}

\looseness=-1 The primary obstructions associated to $(k-1)$-split atlases are precisely the `higher' obstruction classes associated to supermanifolds, as described by Berezin (see \cite[p.~164]{BER}) and more recently by Donagi and Witten in \cite[p.~15]{DW1}. In these texts the higher obstructions were defined inductively as follows. Starting with a $(k-1)$-split atlas $\Xfr^{(k)}$ for a supermanifold~$\Xfr$, one wants to lift this to a $k$-split atlas for~$\Xfr$. The obstruction to doing so is precisely the primary obstruction of~$\Xfr^{(k)}$. The primary obstruction of~$\Xfr^{(k)}$ \emph{is} the $(k-1)$-th obstruction to splitting $\Xfr$. Hence, the $k$-th obstruction to splitting $\Xfr$ is defined if and only if the $(k-1)$-th obstruction to splitting~$\Xfr$ vanishes.

\subsection{The weak splitting problem}
To see the relation of higher obstructions to the splitting problem, note that the groups $\Gc^{(k)}_{(X, E)}$, defined in~\eqref{rhf4hf98498jf8039f33f}, will be trivial for $k$ sufficiently large. Indeed, if $q$ denotes the rank of the holomorphic vector bundle $E\ra X$, then $\Jc_{(X, E)}^k = (0)$ for all $k> q$. Hence $\Gc^{(k)}_{(X, E)} = \{e\}$ for all $k> q$. This leads to the following result, which can also be found (albeit phrased differently) in~\cite{BER,PALAM}:

\begin{THM}\label{rh94hf94hf984jf9j093}A supermanifold $\Xfr$ is split if and only if it admits a $(k-1)$-split atlas for $k> q$.
\end{THM}

We present now a weaker notion of splitting for higher atlases suited to the purposes of this article.

\begin{DEF}\label{rfh974hf9h389fj3}A $(k-1)$-split atlas $\Xfr^{(k)}$, $k>2$, is said to be \emph{weakly non-split} if its associated supermanifold does not coincide with the basepoint. Otherwise, it is said to be \emph{strongly split}. We say $\Xfr^{(k)}$ is split or non-split if its associated supermanifold is split or non-split.
\end{DEF}

Note that non-split will imply \emph{weakly} non-split but not necessarily conversely. Similarly, strongly split will imply split, but not necessarily conversely. These notions lead to the following variant of the splitting problem.

\begin{QUE}\label{rfh89hf98j3f09j039f3}When will a given $(k-1)$-split atlas be weakly non-split?
\end{QUE}

We refer to our proof of Theorem \ref{4fh984fj8jf093j303kf} in Appendix \ref{rgf784g87hf98j3fj0j9f33} for an illustration of the subtleties involved in resolving the splitting problem as outlined in Question~\ref{rfiufh498fj894fj0409}.

\subsection{Exotic atlases}
In \cite{DW1} it was observed that the map $\check{H}^1\big(X, \Gc^{(k)}_{(X, E)}\big)\ra \check{H}^1\big(X, \Gc^{(2)}_{(X, E)}\big)$ will generally fail to be either injective or surjective. This means the existence of an obstructed, $(k-1)$-split atlas for a supermanifold $\Xfr$ does \emph{not} imply $\Xfr$ will be non-split~-- in contrast with the $k = 2$ case in Theorem~\ref{4fh984fj8jf093j303kf}. The following definition captures precisely those obstructed atlases which represent split supermanifolds.

\begin{DEF}\label{rfn4ufh4f8jf039j903jrrr}
A $(k-1)$-split atlas $\Xfr^{(k)}$, for $k> 2$, is said to be \emph{exotic} if it entertains the following two properties:
\begin{enumerate}[(i)]\itemsep=0pt
	\item $\Xfr^{(k)}$ is obstructed;
	\item $\Xfr^{(k)}$ is strongly split.
\end{enumerate}
\end{DEF}

With the sequence in Proposition \ref{rjf984jf894jf90f903k0} we can describe an exotic atlas more concretely. Inspecting the cohomology sequence in \eqref{rfj894jf90jf903k903k0} one step to the left and for the sequence of groups $\Gc^{(k-1)}_{(X, E)} \ra \Gc^{(k)}_{(X, E)} \ra \Qcl^{(k-1)}_{(X, E)}$ gives
\begin{align}
\cdots & \lra H^0\big(X, \Qcl^{(k-1)}_{(X, E)}\big) \stackrel{\al}{\lra} \check{H}^1\big(X, \Gc^{(k)}_{(X, E)}\big) \label{fnvrvknencoeioie}\\
&\stackrel{\be}{\lra} \check{H}^1\big(X, \Gc^{(k-1)}_{(X, E)}\big)\lra H^1\big(X, \Qcl^{(k-1)}_{(X, E)}\big).\label{rh47fh984f09fj903ff}
\end{align}
We prove:

\begin{LEM}\label{4hf84hf89hf8998f3j}Let $\Xfr^{(k)}$ be an obstructed, $(k-1)$-split atlas for $k>2$. If there exists a non-zero $\phi\in H^0\big(X, \Qcl^{(k-1)}_{(X, E)}\big)$ such that $\al(\phi) = \Xfr^{(k)}$, then $\Xfr^{(k)}$ will be exotic.
\end{LEM}

\begin{proof}\sloppy Let $\Xfr^{(k)}$ be a $(k-1)$-split atlas and suppose $\al(\phi) = \Xfr^{(k)}$, for some non-zero $\phi\in H^0\big(X, \Qcl^{(k-1)}_{(X, E)}\big)$. In assuming $\Xfr^{(k)}$ is obstructed, it cannot represent the basepoint in $\check{H}^1\big(X, \Gc^{(k)}_{(X, E)}\big)$. Now by exactness of the sequence in \eqref{fnvrvknencoeioie}--\eqref{rh47fh984f09fj903ff} note, under $\be$, that $\Xfr^{(k)}$ will map to the basepoint in $\check{H}^1\big(X, \Gc^{(k-1)}_{(X, E)}\big)$. Hence it will map to the basepoint in $\check{H}^1\big(X, \Gc^{(2)}_{(X, E)}\big)$. Thus $\Xfr^{(k)}$ is obstructed and strongly split. Hence $\Xfr^{(k)}$ is exotic.
\end{proof}

Further assumptions on the global sections of the obstruction sheaves yield the following sufficient conditions forbidding the existence of exotic atlases.

\begin{PROP}\label{tyhryg7hfiujfoiej}
Suppose that $H^0\big(X, \Qcl^{(k)}_{(X, E)}\big) = (0)$ for all $k>2$. Then there do not exist exotic atlases for any supermanifold modelled on~$(X, E)$.
\end{PROP}

\begin{proof}Suppose $H^0\big(X, \Qcl_{(X, E)}^{(k)}\big) = (0)$. Then from \eqref{fnvrvknencoeioie} we have the exact sequence of pointed sets
\begin{gather}
\{e\} \lra \check{H}^1\big(X, \Gc_{(X, E)}^{(k+1)}\big) \stackrel{\iota_*^{k+1}}{\lra} \check{H}^1\big(X, \Gc_{(X, E)}^{(k)}\big).\label{rhf794hf89hf03j93}
\end{gather}
Unlike for rings or modules, an exact sequence as in \eqref{rhf794hf89hf03j93} above need not imply $\iota_*^{k+1}$ will be injective as a map of sets. But for $\Xfr^{(k+1)}\in \check{H}^1\big(X, \Gc_{(X, E)}^{(k+1)}\big)$ we can nevertheless conclude the following implication from exactness of \eqref{rhf794hf89hf03j93}:
\begin{gather}
\big(\iota_*^{k+1}\big(\Xfr^{(k+1)}\big) = \{e\}\big) \LRa \big(\Xfr^{(k+1)} = \{e\}\big).\label{fh784hfhf983jfj30}
\end{gather}
Assume now that $H^0\big(X, \Qcl_{(X, E)}^{(k)}\big) =(0)$ for all $k$. Then the implication~\eqref{fh784hfhf983jfj30} will hold for all $k$. Let $\Xfr^{(k+1)}\in \check{H}^1\big(X, \Gc_{(X, E)}^{(k+1)}\big)$ and suppose it maps to the basepoint in $ \check{H}^1\big(X, \Gc_{(X, E)}^{(2)}\big)$. It will then follow from~\eqref{fh784hfhf983jfj30} that $\Xfr^{(k+1)}= \{e\}$. Hence~$\Xfr^{(k+1)}$ cannot be obstructed and so, by definition, cannot be exotic.
\end{proof}

\begin{REM}\label{rhf984hf89hf893h8}For $E\ra X$ a holomorphic, rank $q$ vector bundle it is superfluous to assume $H^0\big(X, \Qcl^{(q)}_{(X, E)}\big) = (0)$ in Proposition~\ref{tyhryg7hfiujfoiej}. This is because $\Gc^{(k)}_{(X, E)} = \{e\}$ for all $k> q$, meaning we already have $\{e\}\ra\check{H}^1\big(X, \Gc^{(q)}_{(X, E)}\big)
\ra \check{H}^1\big(X, \Gc^{(q-1)}_{(X, E)}\big)$. That is, Proposition~\ref{tyhryg7hfiujfoiej} will be valid under the slightly weakened assumption that $H^0\big(X, \Qcl^{(k)}_{(X, E)}\big) = (0)$ for all $2 < k < q$.
\end{REM}

\section{Good models}

The existence of an exotic atlas depends essentially on the model $(X, E)$. This leads to the following definition.

\begin{DEF}\label{rh89hf89jf093j9f0fj03}A model $(X, E)$ is said to be \emph{good} if there do not exist any exotic atlases for any supermanifold modelled on $(X, E)$. Otherwise, it is said to support exotic atlases.
\end{DEF}

The prime motivation for introducing good models lies in the following result which follows from the definitions.

\begin{LEM}\label{futhirhhiofoirjfoi}Let $(X, E)$ be a good model. Then if a $(k-1)$-split atlas modelled on $(X, E)$ is obstructed, it will be weakly non-split.
\end{LEM}

It is a meaningful endeavour to understand when a given model $(X, E)$ will be good. In Proposition~\ref{tyhryg7hfiujfoiej} we found sufficient conditions in this vein. We arrive now at our main result in this article, concerning both necessary and sufficient conditions for a model to be good.

\begin{THM}\label{ruihf8h98f983jfj30f}
A model $(X, E)$ will be a good if and only if taking global sections preserves exactness of the sequence
\begin{gather*}
\{e\}\lra \Gc^{(k+1)}_{(X, E)} \lra \Gc^{(k)}_{(X, E)} \lra \Qcl^{(k)}_{(X, E)}\lra \{e\}
\end{gather*}
for all $k\geq 2$.
\end{THM}

\begin{proof}From the exact sequence of sheaves of groups $\Gc^{(k+1)}_{(X, E)} \ra \Gc^{(k)}_{(X, E)} \ra \Qcl^{(k)}_{(X, E)}$ we have the following piece of the long exact sequence on cohomology
\begin{align}
\{e\}& \lra H^0\big(X, \Gc^{(k+1)}_{(X, E)}\big) \lra H^0\big(X, \Gc^{(k)}_{(X, E)}\big) \lra H^0\big(X, \Qcl^{(k)}_{(X, E)}\big)\nonumber\\
&\stackrel{\al}{\lra} \check{H}^1\big(X, \Gc^{(k+1)}_{(X, E)}\big) \lra \cdots.\label{tuirguhf74hf84fjo}
\end{align}
In the direction $(\Leftarrow)$ in the statement of this theorem, assume~\eqref{tuirguhf74hf84fjo} is exact for all $k\geq 2$. Then either $H^0\big(X, \Qcl^{(k)}_{(X, E)}\big) = (0)$ for all $k\geq 2$ or the map $\al$ is constant, sending every $\phi \in H^0\big(X, \Qcl^{(k-1)}_{(X, E)}\big)$ to the basepoint~$\{e\}$. In either case, the proof of Proposition~\ref{tyhryg7hfiujfoiej} will apply since we will obtain exact sequences in \eqref{rhf794hf89hf03j93} for all $k$. In the direction $(\Rightarrow)$, suppose $(X, E)$ is a good model. We will show that $\al$ is constant. Firstly, by definition we know that there will not exist any exotic atlases. Therefore, for any $\phi\in H^0\big(X, \Qcl_{(X, E)}^{(k)}\big)$, the atlas~$\al(\phi)$ must be an unobstructed, $k$-split atlas, for otherwise it would be exotic by Lemma~\ref{4hf84hf89hf8998f3j}. Observe that as a~result we have the following lifting~$\widetilde\al$ of $\al$, represented by the dashed arrow
\begin{gather}\begin{split}&
\xymatrix{
& \check{H}^1\big(X, \Gc^{(k+2)}_{(X, E)}\big)\ar[d]\ar[dr]
\\
H^0\big(X, \Qcl_{(X, E)}^{(k)}\big)\ar@{-->}[ur]^{\widetilde\al}
\ar[r]^\al & \check{H}^1\big(X, \Gc^{(k+1)}_{(X, E)}\big)
\ar[d]^{\om_*}
\ar[r] & \check{H}^1\big(X, \Gc^{(k)}_{(X, E)}\big)
\\
& H^1\big(X, \Qcl^{(k+1)}_{(X, E)}\big).
}\end{split}\label{rj984h98f8jj903jf903}
\end{gather}
The lifting $\widetilde\al$ exists since the vertical maps in \eqref{rj984h98f8jj903jf903} are exact and unobstructedness of $\al(\phi)$ means $\om_*\big(\al(\phi)\big) = 0$. Hence from $\phi$ we obtain a $(k+1)$-split atlas $\widetilde \al(\phi)$. We claim that~$\widetilde\al(\phi)$ is also unobstructed. To see this note that~\eqref{rj984h98f8jj903jf903} will commute and so $\widetilde\al(\phi)$ will map to the basepoint in $\check{H}^1\big(X, \Gc^{(k)}_{(X, E)}\big)$. As such it will map to the basepoint in $\check{H}^1\big(X, \Gc^{(2)}_{(X, E)}\big)$ and so be split. This means, if $\widetilde\al(\phi)$ were obstructed, it would be exotic, contradicting that~$(X, E)$ is a~good model. Hence $\widetilde\al(\phi)$ must be unobstructed. Applying this argument again to $\widetilde\al(\phi)$ will show that $\widetilde\al$ will itself lift to some $\widetilde\al^\p$ valued in $\check{H}^1\big(X, \Gc^{(k+3)}_{(X, E)}\big)$ and that to $\phi$ will be associated an unobstructed, $(k+2)$-split atlas. We can proceed inductively now and keep lifting the map~$\al$ to get unobstructed, $k^\p$-split atlases for any $k^\p> k$, resulting in the commutative diagram
\begin{gather}\begin{split}&
\xymatrix{
& \check{H}^1\big(X, \Gc^{(k^\p)}_{(X, E)}\big)\ar[d]
\\
H^0\big(X, \Qcl_{(X, E)}^{(k)}\big)\ar@{-->}[ur]^{\widetilde\al^{k^\p}}
\ar[r]^\al & \check{H}^1\big(X, \Gc^{(k+1)}_{(X, E)}\big).
}\end{split}\label{rfh49fh948fj984jf0}
\end{gather}
Since $\Gc^{(k^\p)}_{(X, E)}$ is trivial for $k^\p$ sufficiently large, commutativity of~\eqref{rfh49fh948fj984jf0} requires~$\al$ map~$\phi$ to the basepoint~$\{e\}$ for all~$\phi$. Thus~$\al$ is constant. This argument is independent of $k\geq 2$ and depends only on $(X, E)$ being a good model. The theorem now follows.
\end{proof}

\begin{REM}Concerning the diagram \eqref{rj984h98f8jj903jf903}, it is argued in \cite[Theorem~4.7.1, pp.~163--164]{BER} and in~\cite[p.~16]{DW1} that there will always exist a lift $\widetilde\al$ of $\al$ when $k$ is odd.
\end{REM}

\section{Illustrations}\label{rfh94hf98hf984jf9j90k3}

To convince the reader that exotic atlases and good models exist we present some illustrations. It will be convenient to firstly present the following more explicit characterisation of the obstruction sheaves which we will make use of in our applications. A justification can be found in~\cite{GREEN}.

\begin{LEM}\label{fufh48hf89j4f09j09f3p}To any model $(X, E)$ there exists an isomorphism of sheaves
\begin{gather}
\Qcl^{(k)}_{(X, E)}=
\begin{cases}
\wedge^k\Ec\otimes T_X, & \text{$k$ is even},\\
\Ec^*\otimes \wedge^k\Ec, & \text{$k$ is odd}.
\end{cases}\label{rjf9804jf93jf903jf093j03}
\end{gather}
\end{LEM}

\subsection{On Riemann surfaces}
Let $X$ be a complex manifold and $E\ra X$ a holomorphic vector bundle of rank $3$. Then $\Gc_{(X, E)}^{(k)} = \{e\}$ for all $k> 3$. With the identifications of the obstruction sheaves in~\eqref{rjf9804jf93jf903jf093j03} we obtain the following long exact sequence, corresponding to the more general case in~\eqref{tuirguhf74hf84fjo}
\begin{align}
\{e\}& \lra H^0\big(X, \Gc_{(X, E)}^{(3)}\big) \lra H^0\big(X, \Gc_{(X, E)}^{(2)}\big) \lra H^0\big(X, \wedge^2\Ec\otimes T_X\big)\nonumber\\
&\stackrel{\al}{\lra}\check{H}^1\big(X, \Gc_{(X, E)}^{(3)}\big)\lra\cdots.\label{jghhthrjdkfhnrjeo8474}
\end{align}
Theorem \ref{ruihf8h98f983jfj30f} says $(X, E)$ will be a good model if and only if $\al$ in \eqref{jghhthrjdkfhnrjeo8474} vanishes. This is generally a non-trivial condition to check. It will be true however if $h^0\big(X, \wedge^2\Ec\otimes T_X\big) = 0$.\footnote{For any abelian sheaf $\Ac$ on $X$ we denote $h^i(X, \Ac) := \dim H^i(X, \Ac)$.}

\begin{PROP}\label{rhf4hf84hf89jf893}Let $X$ be a Riemann surface of genus $g$ and $E\ra X$ a rank $3$, holomorphic vector bundle. Suppose it splits into a sum of non-negative, holomorphic line bundles and $\deg E < 3g-3$. Then $(X, E)$ will be a good model.
\end{PROP}

\begin{proof}It suffices to show $h^0\big(X, \wedge^2\Ec\otimes T_X\big) = 0$. Let $X$ be a genus $g$ Riemann surface. Assu\-ming~$E$ is holomorphically split we can write $E = L_{d_1}\oplus L_{d_2}\oplus L_{d_3}$, for $L_{d_i}\ra X$ a line bundle on~$X$ of degree~$d_i$. If $L_{d_i}$ is non-negative, then $d_i \geq 0$. Note that $\wedge^2E = L_{d_1+d_2} \oplus L_{d_1+d_3} \oplus L_{d_2+d_3}$; and $\deg \wedge^2E = 2\deg E = 2(d_1+d_2+d_3)$. With $\deg T_X = 2 - 2g$ and the above description of~$\wedge^2E$ we see that $h^0\big(\wedge^2\Ec\otimes T_X\big) = 0$ when
\begin{gather*}
d_1 + d_2 < 2g - 2, \qquad d_1 + d_3 < 2g - 2 \qquad \text{and}\qquad d_2 + d_3 < 2g - 2.
\end{gather*}
These conditions are equivalent to $\deg E < 3g-3$ when $d_i\geq 0$ for each $i$.
\end{proof}

\subsection{In genus zero}
We can be considerably more specific on the projective line.

\begin{THM}\label{rfh794hf984f98j}Let $E = L_{d_1}\oplus L_{d_2}\oplus L_{d_3}$ be a rank~$3$, holomorphic vector bundle on $\Pbb^1_\Cbb$, where~$L_{d_i}$ is a line bundle on~$\Pbb^1_\Cbb$ of degree~$d_i$. Then $(\Pbb_{\Cbb}^1, E)$ will be a good model if and only if at least one of~$d_i \neq -1$.
\end{THM}

Before giving a proof of Theorem~\ref{rfh794hf984f98j} we provide some preliminary remarks about exotic atlases. Regarding their existence, this was addressed in \cite[pp.~15--16]{DW1} for the present situation, i.e., when $E$ has rank $3$. We comment on this here, but in a little more generality. Let~$E\ra X$ be a holomorphic vector bundle of rank $q$. Note that $\Gc_{(X, E)}^{(k)} = \{e\}$ for all $k > q$. Hence $\Gc^{(q)}_{(X, E)} \cong \Qcl_{(X, E)}^{(q)}$ which implies $\check{H}^1\big(X, \Gc_{(X, E)}^{(q)}\big) \stackrel{\sim}{\ra} H^1\big(X, \Qcl_{(X, E)}^{(q)}\big)$. With the exact sequence
\begin{gather*}
\cdots \lra H^0\big(X, \Qcl^{(q-1)}_{(X, E)}\big) \lra \check{H}^1\big(X, \Gc_{(X, E)}^{(q)}\big) \lra \check{H}^1\big(X, \Gc_{(X, E)}^{(q-1)}\big) \lra \cdots
\end{gather*}
we deduce:

\begin{LEM}\label{rhf84h94h89484409}Let $E\ra X$ be a rank $q$, holomorphic vector bundle. Any $(q-1)$-split atlas in $\check{H}^1\big(X, \Gc_{(X, E)}^{(q)}\big)$ in the image of a non-trivial section in $H^0\big(X, \Qcl^{(q-1)}_{(X, E)}\big)$ will be obstructed and hence, by Lemma~{\rm \ref{4hf84hf89hf8998f3j}}, exotic.
\end{LEM}

We now present a proof of Theorem \ref{rfh794hf984f98j}.

\begin{proof}[Proof of Theorem~\ref{rfh794hf984f98j}] Write $\Ec = \Oc(d_1) \oplus\Oc(d_2)\oplus\Oc(d_3)$, where $\Ec$ denotes the sheaf of sections of $E$. Then
\begin{gather*}
\Qcl^{(2)}_{(\Pbb^1_\Cbb, E)}= \Oc(d_1 + d_2 + 2)\oplus \Oc(d_1 + d_3 + 2)\oplus\Oc(d_2 + d_3 + 2) \qquad \text{and}\\
\Qcl_{(\Pbb^1_\Cbb, E)}^{(3)}= \Oc(d_1 + d_2)\oplus \Oc(d_1 + d_3)\oplus\Oc(d_2 + d_3).
\end{gather*}
In order to construct an exotic atlas for a supermanifold modelled on $\big(\Pbb^1_\Cbb, E\big)$ it will be necessary for
\begin{enumerate}[(i)]\itemsep=0pt
	\item $h^0\big(\Qcl^{(2)}_{(\Pbb^1_\Cbb, E)} \big)\neq 0$ (by Proposition~\ref{tyhryg7hfiujfoiej} and Remark~\ref{rhf984hf89hf893h8}) and
	\item $h^1\big(\Qcl_{(\Pbb^1_\Cbb, E)}^{(3)}\big) \neq0$ (since $\Gc^{(3)}_{(\Pbb^1_\Cbb, E)}\cong \Qcl_{(\Pbb^1_\Cbb, E)}^{(3)}$ here).
\end{enumerate}
From Bott's formula on the dimension of the cohomology of line bundles on projective space,\footnote{See, e.g., \cite[p.~4]{OSS}.} both~(i) and~(ii) will be satisfied if and only if $(d_1, d_2, d_3) = (-1,-1, -1)$. Now in this case, where $(d_1, d_2, d_3) = (-1,-1, -1)$, we have by Serre duality
\begin{gather*}
 H^0\big(\Pbb_\Cbb^1, \Qcl^{(2)}_{(\Pbb^1_\Cbb, E)}\big)\cong H^1\big(\Pbb_\Cbb^1, \Qcl^{(3)}_{(\Pbb^1_\Cbb, E)}\big).
\end{gather*}
Upon observing $H^1\big(\Pbb_\Cbb^1, \Qcl^{(3)}_{(\Pbb^1_\Cbb, E)}\big)\cong \check{H}^1\big(\Pbb^1_\Cbb, \Gc^{(3)}_{(\Pbb_\Cbb^1, E)}\big)$, the theorem will then follow from Lem\-ma~\ref{rhf84h94h89484409} applied to $q = 3$.
\end{proof}

We will now consider those bundles of higher rank which decompose into copies of a single line bundle, i.e., $E = \oplus^qL_d$. As we will see, the case $d = -1$ will be similar to that in Theorem~\ref{rfh794hf984f98j}.

\begin{THM}\label{hf78hf984f8j40f094}Let $E\ra \Pbb^1_\Cbb$ be a rank $q$, holomorphic vector bundle. Suppose $E = \oplus^qL_d$ where $L_d\ra \Pbb_\Cbb^1$ is a line bundle of degree $d$. Then
\begin{enumerate}[$(i)$]\itemsep=0pt
	\item if $d = -1$ then $\big(\Pbb^1_\Cbb, E\big)$ will support exotic atlases;
	\item if $d < -1$, then $\big(\Pbb^1_\Cbb, E\big)$ will be a good model.
\end{enumerate}
\end{THM}

\begin{proof}\looseness=-1 Let $\Oc(d)$ denote the sheaf of sections of $L_d$ so that $\Ec = \oplus^q\Oc(d)$. We begin by pro\-ving~(ii). We have generally $\wedge^k\Ec = \oplus^{\binom{q}{k}}\Oc(kd)$. As such the obstruction sheaves are given by
\begin{gather}
\Qcl_{(\Pbb_\Cbb^1, E)}^{(k)}=
\begin{cases}
\oplus^{\binom{q}{k}} \Oc(kd+2), &\text{$k$ is even},\\
\oplus^{k\binom{q}{k}}\Oc\big((k-1)d\big), &\text{$k$ is odd}.
\end{cases}\label{rfj984hf894jf90f903}
\end{gather}
Note that $2\leq k\leq q$ here. Evidently, if $d< -1$ the obstruction sheaves will be sums of line bundles of negative degree. Hence $H^0\big(\Pbb_\Cbb^1, \Qcl_{(\Pbb_\Cbb^1, E)}^{(k)}\big) = (0)$. Part (ii) then follows from Proposition~\ref{tyhryg7hfiujfoiej}. As for part~(i), firstly note that
\begin{align}
\wedge^\bt \Ec &= \wedge^0\Ec \oplus \wedge^1\Ec\oplus \wedge^2\Ec \oplus\cdots\oplus \wedge^q\Ec\notag\\
&=\Oc(0)\oplus \big[\oplus^q\Oc(d)\big] \oplus\big[\oplus^{\binom{q}{2}}\Oc(2d)\big]\oplus \cdots \oplus \Oc(qd).
\notag
\end{align}
We have $\operatorname{{\mathcal A}ut} \Ec\!\subset\! \oplus^{q^2}\Oc(0)$ and $\operatorname{{\mathcal A}ut}\wedge^\bt \Ec\!\subset\! \oplus_{0\leq a\leq b\leq q}\Oc((b-a)d)$. Note in particular that~\smash{$b-a\geq 0$}. Regarding the sheaf $\Gc_{(\Pbb_\Cbb^1, E)}^{(k)}$ observe that it can be realised as a subsheaf of $\oplus_{0\leq a< b\leq q}\Oc((b-a)d)$ where now $b-a \geq k-1$. Therefore if $d < 0$:
\begin{gather}
H^0\big(\Pbb_\Cbb^1, \Gc_{(\Pbb_\Cbb^1, E)}^{(k)}\big) = \{1\},\label{rfh784hfhf983jf9030}
\end{gather}
for all $k\geq 2$. Now if $d = -1$ then $h^0\big(\Pbb^1_\Cbb, \Qcl_{(\Pbb_\Cbb^1, E)}^{(2)}\big) =\binom{q}{2}\neq0$ from~\eqref{rfj984hf894jf90f903}. Since $H^0\big(\Pbb_\Cbb^1, \Gc_{(\Pbb_\Cbb^1, E)}^{(2)}\big) = \{e\}$ from~\eqref{rfh784hfhf983jf9030} we have the exact sequence
\begin{gather*}
\{1\}\lra H^0\big(\Pbb^1_\Cbb, \Qcl_{(\Pbb_\Cbb^1, E)}^{(2)}\big) \stackrel{\al}{\lra} \check{H}^1\big(\Pbb^1_\Cbb, \Gc_{(\Pbb_\Cbb^1, E)}^{(3)}\big),
\end{gather*}
where $\al$ is the boundary map from~\eqref{jghhthrjdkfhnrjeo8474}. In particular $\al$ does not vanish and any atlas in its image will be exotic.
\end{proof}

\subsection{In genus one}
Here we can relax some of conditions in Proposition~\ref{rhf4hf84hf89jf893}, although we cannot be as general as in the genus zero case.

\begin{PROP}\label{rhf98hf74hf984fj039}Let $C$ be a Riemann surface of genus one and $E\ra C$ a holomorphic, rank~$3$ vector bundle. Then $(C, E)$ will be a good model when either:
\begin{enumerate}[$(i)$]\itemsep=0pt
	\item $\deg \Qcl_{(C, E)}^{(2)} = h^0\big(\Qcl_{(C, E)}^{(2)}\big)$ or
	\item $\deg \Qcl_{(C, E)}^{(2)} = -h^1\big(\Qcl_{(C, E)}^{(2)}\big)$.
\end{enumerate}
\end{PROP}

\begin{proof}In genus one note that $T_C = \Oc_C$ is a line bundle of degree zero. Hence we have a~natural isomorphism $\Qcl_{(C, E)}^{(2)} \cong \Qcl_{(C, E)}^{(3)}$. The proposition now follows from the Riemann--Roch theorem.
\end{proof}

\appendix

\section{Proof of Theorem \ref{4fh984fj8jf093j303kf}}\label{rgf784g87hf98j3fj0j9f33}

We will firstly present an erroneous proof as it will be instructive in illustrating the subtlety involved in the splitting problem.

\begin{proof}[False proof] Let $\Xfr$ be a supermanifold modelled on~$(X, E)$ and suppose its primary obstruction~$\om_*(\Xfr)$ does not vanish. The map~$\om_*\colon \check{H}^1\big(X, \Gc^{(2)}_{(X, E)}\big) \ra H^1\big(X, \Qcl_{(X, E)}^{(2)}\big)$ is a map of pointed sets, sending the basepoint $\{e\}\in \check{H}^1\big(X, \Gc^{(2)}_{(X, E)}\big)$ to the basepoint $0\in H^1\big(X, \Qcl_{(X, E)}^{(2)}\big)$. Since $\om_*(\Xfr) \neq 0$, $\Xfr$ cannot be split.
\end{proof}

The fault in the above reasoning lies in our failure to consider isomorphisms induced by the global symmetries $\Aut(E)$ acting on $\check{H}^1\big(X, \Gc^{(2)}_{(X, E)}\big)$ (cf.\ Definition~\ref{rfhhf983fj390f3} and Theorem~\ref{rh9r8h984hf89jf0fj3}). If $\Aut(E)$ fixes the basepoint $\{e\}$ then the above argument would be valid, but there is of course no reason for it to fix the basepoint in general. Consider instead the following
\begin{gather}\begin{split}&
\xymatrix{
\cdots\ar[r] & \Aut(E) \ar[r]^\dt& \check{H}^1\big(X, \Gc^{(2)}_{(X, E)}\big) \ar[d]_{\om_*}\ar[r] & \check{H}^1\big(X, \operatorname{{\mathcal A}ut}\wedge^\bt \Ec\big) \ar[r] & \cdots
\\
& & H^1\big(X, \Qcl_{(X, E)}^{(2)}\big).& &
}\end{split} \label{rfh74hf983jfj30}
\end{gather}
Since \eqref{rfh74hf983jfj30} is exact we have: if a supermanifold $\Xfr$ is split, then $\Xfr = \dt(\phi)$ for some $\phi\in \Aut(E)$. Observe that Theorem~\ref{4fh984fj8jf093j303kf} will then follow from:

\begin{PROP}\label{rfg784gf98hf89j3f03}The composition of maps $\om_*\dt$ in \eqref{rfh74hf983jfj30} vanishes.
\end{PROP}

The proof we submit of Proposition \ref{rfg784gf98hf89j3f03} is based on the discussion in \cite[p.~16]{DW1}. We will cite the following result whose proof can be found in~\cite{ONISHCLASS}.

\begin{LEM}\label{rfu4f9h38f03jf93jf93j}Let $(X, E)$ be a model. The subgroup $\Cbb^\times\cdot 1_E< \Aut(E)$ acts on the $k$-th obstruction space $H^1\big(X, \Qcl_{(X, E)}^{(k)}\big)$ by sending $v \mapsto \lam^kv$ for any vector $v\in H^1\big(X, \Qcl_{(X, E)}^{(k)}\big)$.
\end{LEM}

\begin{proof}[Proof of Proposition~\ref{rfg784gf98hf89j3f03}] \sloppy From Theorem~\ref{groeiveytvdbno} we know that $\Aut(E)$ will act on $\check{H}^1\big(X, \Gc^{(2)}_{(X, E)}\big)$. Denote this action by $\star$. Our first claim is that the subgroup $\Cbb^\times\cdot 1_E$ will act trivially on the image of $\Aut(E)$. To see this, let $(U_i\ra X)$ be an open cover and $\phi \in \Aut(E)$. If $\widetilde\phi_i\in \big(\operatorname{{\mathcal A}ut}\wedge^\bt\Ec\big)(U_i)$ are local lifts of $\phi_i$, then
\begin{gather*}
\dt(\phi)_{ij}= \widetilde\phi_i \widetilde\phi_j^{-1}.
\end{gather*}
Now clearly any $\lam\cdot 1_E\in \Cbb^\times \cdot 1_E< \Aut(E)$ will act by sending $\phi_i \mapsto \lam \phi_i$. Hence,
\begin{gather}
\big(\lam\cdot 1_E\big)\star \dt(\phi)_{ij} = \big(\lam \widetilde\phi_i\big)\big(\lam^{-1}\widetilde\phi_j^{-1}\big) =\widetilde\phi_i \widetilde\phi_j^{-1}=\dt(\phi)_{ij}.\label{rhf894hf8h09fj30jf930}
\end{gather}
The group $\Aut(E)$ will act on $\check{H}^1\big(X, \Gc^{(2)}_{(X, E)}\big)$ and this action will induce an action on $H^1\big(X, \Qcl_{(X, E)}^{(2)}\big)$. The action of $\Cbb^\times\cdot 1_E$ mentioned in Lemma~\ref{rfu4f9h38f03jf93jf93j} is compatible with this in\-du\-ced $\Aut(E)$ action. In particular we can conclude: for a supermanifold $\Xfr\in \check{H}^1\big(X, \Gc^{(2)}_{(X, E)}\big)$ and any $\lam\cdot 1_E\in \Cbb^\times\cdot1_E$,
\begin{gather}
\om_*\big((\lam\cdot 1_E)\star \Xfr\big)=\lam^2\om_*\big(\Xfr\big).\label{rhf893h8j30fj390}
\end{gather}
In comparing \eqref{rhf894hf8h09fj30jf930} and \eqref{rhf893h8j30fj390} we see that if $\Xfr = \dt(\phi)$, we must have $\om_*\big(\Xfr\big) = 0$. This proves Proposition~\ref{rfg784gf98hf89j3f03} from whence Theorem~\ref{4fh984fj8jf093j303kf} follows.
\end{proof}
\vspace{-1.5mm}

\subsection*{Acknowledgements}
The author would like to acknowledge the helpful feedback of the anonymous referees.

\vspace{-1.5mm}

\pdfbookmark[1]{References}{ref}
\LastPageEnding

\end{document}